\documentclass{amsart}
\usepackage{graphicx,verbatim, amsmath, amssymb, amsthm, amsfonts, epsfig, ifthen,mathtools,mathabx,enumerate} 

\newtheorem{proposition}{Proposition}[section]
\newtheorem{theorem}[proposition]{Theorem}

\newtheorem{prop}[proposition]{Proposition}
\newtheorem{cor}[proposition]{Corollary}

\theoremstyle{definition}
\newtheorem{example}[proposition]{Example}

\newtheorem{question}[proposition]{Question}

\theoremstyle{remark}
\newtheorem{remark}[proposition]{Remark}

\numberwithin{equation}{section}
\usepackage[usenames]{color}

\usepackage[usenames]{color}

\newcounter{margincounter}
\newcommand{\newword}[1]{\textbf{\emph{#1}}}

\newcommand{\eulerian}{\genfrac{\langle}{\rangle}{0pt}{}}

\newcommand{\set}[1]{\left\lbrace#1\right\rbrace}
\renewcommand{\Join}{\bigvee}
\newcommand{\inv}{\operatorname{inv}}

\newcommand{\meet}{\wedge}
\newcommand{\join}{\vee}
\newcommand{\E}{\mathcal{E}}
\newcommand{\F}{\mathcal{F}}
\newcommand{\G}{\mathcal{G}}
\newcommand{\pidown}{\pi_\downarrow}
\newcommand{\piup}{\pi^\uparrow}

\title[Noncrossing arc diagrams]{Noncrossing arc diagrams and canonical join representations}
\author{Nathan Reading}
\thanks{Research was conducted while the author was partially supported by NSF grant DMS-1101568.}
\subjclass[2010]{}

\begin{document}

\begin{abstract}
We consider two problems that appear at first sight to be unrelated.
The first problem is to count certain diagrams consisting of noncrossing arcs in the plane.
The second problem concerns the weak order on the symmetric group.
Each permutation $x$ has a canonical join representation: a unique lowest set of permutations joining to $x$.
The second problem is to determine which sets of permutations appear as canonical join representations.
The two problems turn out to be closely related because the noncrossing arc diagrams provide a combinatorial model for canonical join representations.
The same considerations apply more generally to lattice quotients of the weak order.
Considering quotients produces, for example, a new combinatorial object counted by the Baxter numbers and an analogous new object in bijection with generic rectangulations.
\end{abstract}
\maketitle


\section{Noncrossing arc diagrams}\label{diagrams sec}
The key objects in this paper are  \newword{noncrossing arc diagrams}, certain diagrams consisting of arcs satisfying certain rules, including the requirement that arcs not cross.
Each diagram begins with $n$ distinct points on a vertical line.
We identify the points with the numbers $1,\ldots,n$, with $1$ at the bottom.
Each diagram is consists of some (or no) curves called \newword{arcs} connecting the points.
Each arc must satisfy the following requirement:
\begin{enumerate}
\item[(A)] 
The arc connects a point $p$ to a strictly higher point $q$, moving monotone upwards from $p$ to $q$ and passing either to the left or to the right of each point between $p$ and $q$.
The arc may pass to the left of some points and to the right of others.
\end{enumerate}
The diagram must also satisfy the following two pairwise compatibility conditions:
\begin{enumerate}
\item[(C1)]
No two arcs intersect, except possibly at their endpoints.
\item[(C2)]\setcounter{enumi}{2} \label{diag def}
No two arcs share the same upper endpoint or the same lower endpoint.
\end{enumerate}
\begin{figure}[ht]
\scalebox{0.9}{\includegraphics{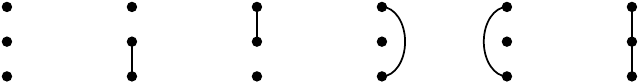}}
\caption{Noncrossing arc diagrams on $3$ points}
\label{diagrams3}
\end{figure}
Each noncrossing arc diagram determines some combinatorial data, namely which pairs of points are joined by an arc and which points are left and right of each arc.
Two noncrossing arc diagrams are considered to be combinatorially equivalent if they determine the same combinatorial data, and we consider arcs and noncrossing arc diagrams only up to combinatorial equivalence.

For $n=1$, there is one noncrossing arc diagram (with no arcs), and for $n=2$ there are two noncrossing arc diagrams (with one arc or no arcs).
Figures~\ref{diagrams3} and~\ref{diagrams4} show the $6$ noncrossing arc diagrams on $3$ points and the $24$ noncrossing arc diagrams on $4$ points.
\begin{figure}
\scalebox{0.82}{\includegraphics{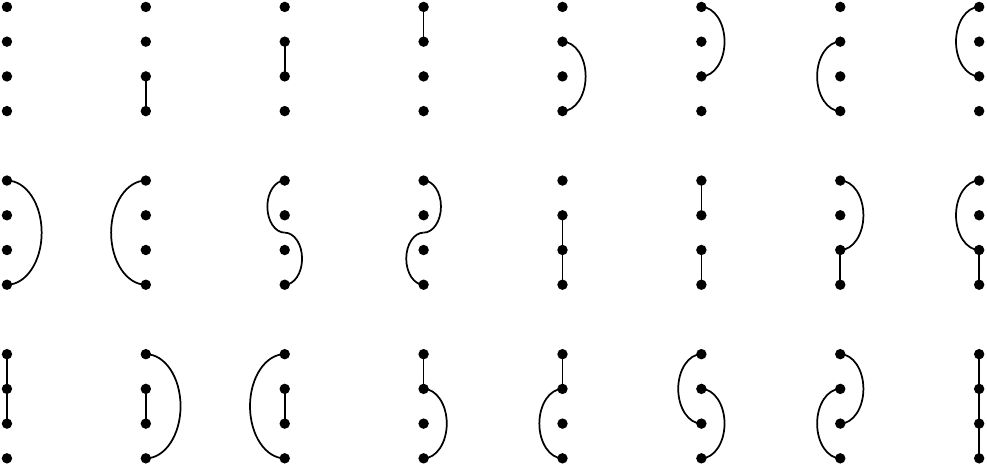}}
\caption{Noncrossing arc diagrams on $4$ points}
\label{diagrams4}
\end{figure}

Say two arcs are \newword{compatible} if there exists a noncrossing arc diagram containing those two arcs.
Compatibility is a combinatorial condition, depending only on the endpoints of the two curves and on which points are left and right of each curve:
Given arcs $\alpha_1$ and $\alpha_2$, suppose there is a point $p$ that is left of (or an endpoint of) $\alpha_1$ and is right of (or an endpoint of) $\alpha_2$, with $p$ not an endpoint of both arcs.
Then a noncrossing arc diagram containing $\alpha_1$ and $\alpha_2$ must have $\alpha_1$ to the right of $\alpha_2$.
The two curves are compatible if and only they don't share upper endpoints or lower endpoints and if there do not exist both a point $p$ forcing $\alpha_1$ to be right of $\alpha_2$ and a point $p'$ forcing $\alpha_2$ to be right of $\alpha_1$.
Given any noncrossing arc diagram, the arcs in the diagram are pairwise compatible.
However, \textit{a priori} we don't know, given a collection of pairwise compatible arcs, if it is possible to fix a representative for each arc so that the \emph{representatives} satisfy (C1) pairwise.
In Section~\ref{proof sec}, we prove that it \emph{is} possible to fix such a set of representatives.

One family of noncrossing arc diagrams is very familiar.
A \newword{left arc} is an arc that does not pass to the right of any point, and we will call a noncrossing arc diagram having only left arcs a \newword{left noncrossing arc diagram}.
The left noncrossing arc diagrams are more commonly known as \newword{noncrossing partitions}, although it is more typical to draw points on a vertical line and allow only arcs that don't pass below any points.
The enumeration of noncrossing partitions is well-known, and we give it here in the language of this paper:
The number of left noncrossing arc diagrams on $n$ points is the Catalan number $C_n=\frac1{n+1}\binom{2n}n$.
The number of such diagrams with $k$ arcs is the Narayana number $\frac1n\binom n k\binom n{k-1}$.
One might similarly define a \newword{right arc} to be an arc that does not pass to the left of any point, and the same enumerative statements hold with ``left'' replaced by ``right.''
Surprisingly, something nice also happens when we mix right arcs with left arcs:

\begin{theorem}\label{diagrams left right}
The number of noncrossing arc diagrams on $n$ points having only left arcs and right arcs is the Baxter number 
\[B(n)=\binom{n+1}{1}^{-1}\binom{n+1}{2}^{-1}\,\,\sum_{k=0}^{n-1}\binom{n+1}{k}\binom{n+1}{k+1}\binom{n+1}{k+2}.\]
\end{theorem}

Among several combinatorial objects already known to be counted by the Baxter number $B(n)$ are the \newword{diagonal rectangulations} with $n$ rectangles.  
(This is due to \cite{ABP,CGHK}.  See \cite[Section~1]{rectangle} and \cite[Remark~6.6]{rectangle} for details on attribution.)
Fixing a square and one diagonal of the square, diagonal rectangulations are (combinatorial types of) decompositions of the square into rectangles such that each rectangle's interior intersects the given diagonal.
A larger set of rectangulations are the \newword{generic rectangulations}.
(See~\cite{generic}.)
These are (combinatorial types of) decompositions of a square into rectangles such that no four rectangles have a common corner.
Just as the diagonal rectangulations are in bijection with a certain set of noncrossing arc diagrams by Theorem~\ref{diagrams left right}, the generic rectangulations are in bijection with a larger set of noncrossing arc diagrams.
An \newword{inflection} of an arc in a diagram is a pair of adjacent points with one left of the arc and the other right of the arc.
If we draw arcs in the natural way as in Figures~\ref{diagrams3} and~\ref{diagrams4}, the arc has an inflection point (in the sense of curvature) between the two points comprising the inflection.
The noncrossing arc diagrams of Theorem~\ref{diagrams left right} are diagrams whose arcs have no inflections.
A similarly nice thing happens if we allow up to one inflection in each arc.

\begin{theorem}\label{diagrams inflection}
The noncrossing arc diagrams on $n$ points with each arc having at most one inflection are in bijection with generic rectangulations with $n$ rectangles.
\end{theorem}

Noncrossing partitions, diagonal rectangulations, and generic rectangulations are all in bijection with certain pattern-avoiding permutations.
These connections and the evidence for small $n$ suggest the following theorem, which is somewhat surprising \emph{a priori}.
\begin{theorem}\label{diagrams perm}
There are $n!$ noncrossing arc diagrams on $n$ points.
\end{theorem}
We prove Theorems~\ref{diagrams left right}, \ref{diagrams inflection}, and~\ref{diagrams perm} and several other enumerative statements in Sections~\ref{proof sec} and~\ref{quot sec}.
Before doing that, we explain canonical join representations of permutations in Section~\ref{can sec}.
Canonical join representations are not really necessary for the proofs of Theorems~\ref{diagrams left right}, \ref{diagrams inflection}, and~\ref{diagrams perm}.
However, canonical join representations explain ``what's really going on'' in these theorems and in many of the other results.
Furthermore, noncrossing arc diagrams provide the answer to a very natural question about canonical join representations (Question~\ref{can q}).

Noncrossing arc diagrams are foreshadowed in the work of Bancroft~\cite{Bancroft} and Petersen~\cite{Petersen} on shard intersections in the Coxeter arrangement of type A.

\section{Canonical join representations of permutations}\label{can sec}
A \newword{join representation} for an element $x$ in a finite lattice $L$ is an identity $x=\Join S$, where $S$ is a subset of $L$.
The join representation $x=\Join S$ is \newword{irredundant} if there is no proper subset $S'\subset S$ such that $x=\Join S'$.
In particular, if $\Join S$ is irredundant, then $S$ is an antichain.
We define a relation $\ll$ on subsets of $L$ by setting $S\ll T$ if, for every $s\in S$, there exists a $t\in T$ with $s\le t$.
The relation $\ll$ restricts to a partial order on antichains (equivalent to containment order on order ideals generated by the antichains).
A join representation $x=\Join S$ is called the \newword{canonical join representation} of $x$ if it is irredundant and if every join representation $x=\Join T$ has $S\ll T$.
In other words, $S$ is the unique minimal antichain, in the partial order $\ll$, among antichains joining to $x$.
If $x=\Join S$ is  the canonical join representation of $x$, then the elements of $S$ are called the \newword{canonical joinands} of $x$.
We sometimes abuse terminology by referring to the set $S$ itself, rather than the expression $x=\Join S$, as the canonical join representation of~$x$.

If an element has a canonical join representation, then each canonical joinand is join-irredicible.   
That is, each canonical joinand $j$ covers exactly one element of $L$, or equivalently, there is no join representation for $j$ consisting of elements strictly below $j$.
The following proposition is immediate.
\begin{prop}\label{ji can}
An element $j$ of a finite lattice $L$ is join-irreducible if and only if its only canonicial joinand is $j$ itself.
\end{prop}

A lattice in which every element has a canonical join representation is called join-semidistributive.  If the dual condition (every element has a canonical meet representation) holds as well, then the lattice is called \newword{semidistributive}.
The usual definition of semidistributivity involves two dual weakenings of the distributive laws, but the present definition is equivalent by \cite[Theorem~2.24]{FreeLattices}.

Suppose $L$ is a finite join-semidistributive lattice.
Not every antichain of join-irreducible elements in $L$ is a canonical join representation.
We define the \newword{canonical join complex} of $L$ to be the abstract simplicial complex whose vertices are the join-irreducible elements of $L$ and whose faces are the sets $S$ such that $\Join S$ is a canonical join representation.
This is a simplicial complex in light of Proposition~\ref{ji can} and the following proposition.
\begin{prop}
Suppose $L$ is a finite lattice and let $S$ be a subset of $L$.
If $\Join S$ is a canonical join representation and $S'\subseteq S$ then $\Join S'$ is also a canonical join representation.
\end{prop}
\begin{proof}
Let $x$ be the element $\Join S$, let $x'$ be the element $\Join S'$, and write $C$ for $S\setminus S'$.
Suppose $\Join T'$ is a join representation for $x'$ and let $s\in S'$.
We need to show that there exists $t\in T'$ with $s\le t$.
Writing $T=T'\cup C$, the expression $\Join T$ is a join representation for $x$.
Since $\Join S$ is the canonical join representation for $x$, there exists $t\in T$ such that $s\le t$.
Since $S$ is an antichain and $s\le t$, we see that $t\not\in C$.
Thus $t\in T'$ as desired.
\end{proof}

A permutation $x$ of $\set{1,\ldots,n}$ is a sequence $x_1x_2\cdots x_n$ such that $\set{x_1,\ldots,x_n}=\set{1,\ldots,n}$.
The weak order on permutations is a partial order whose cover relations are $x_1\cdots x_n\lessdot y_1\cdots y_n$ whenever there exists $i$ such that $x_i=y_{i+1}<y_i=x_{i+1}$ and such that $x_j=y_j$ for $j\not\in\set{i,i+1}$.
We write $S_n$ for the set of permutations of $\set{1,\ldots,n}$ partially ordered under the weak order.
Figure~\ref{weak ex} shows this partial order for $n=3$ and for $n=4$.
\begin{figure}
\scalebox{0.8}{\includegraphics{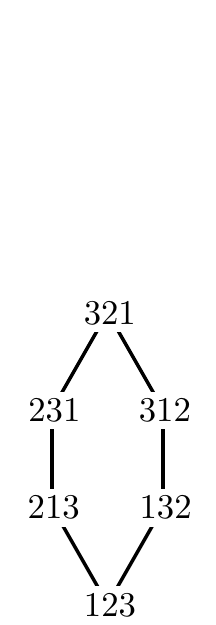}}\qquad\qquad\scalebox{0.8}{\includegraphics{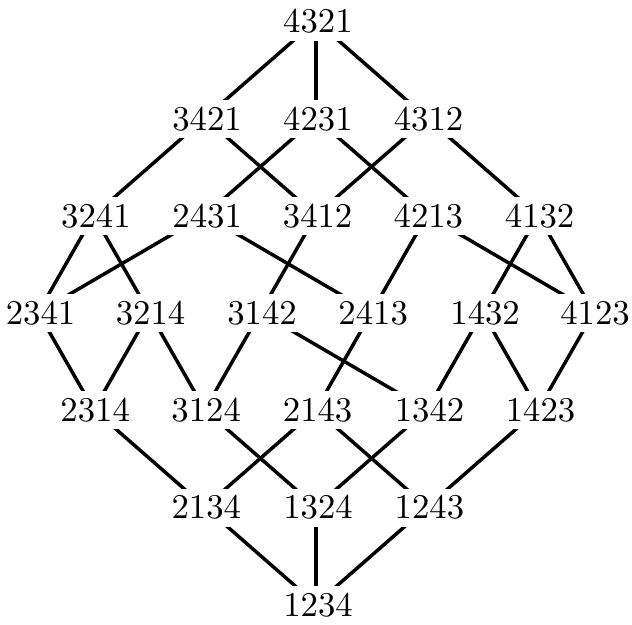}}
\caption{The weak order on $S_n$ for $n=3$ and for $n=4$}
\label{weak ex}
\end{figure}

An \newword{inversion} in $x_1\cdots x_n$ is a pair $(x_i,x_j)$ with $i<j$ and $x_i>x_j$.
For example, the inversions of $25314$ are $(2,1)$, $(3,1)$, $(5,1)$, $(5,3)$ and $(5,4)$.
Write $\inv(x)$ for the \newword{inversion set} (the set of inversions) of $x$.
The weak order on permutations is characterized by containment of inversion sets.
That is, $x\le y$ in the weak order if and only if $\inv(x)\subseteq\inv(y)$.

The weak order on permutations is a lattice, and Duquenne and Cherfouh showed that it is semidistributive \cite[Theorem~3]{DuCh}.
(More generally, the weak order on any finite Coxeter group is semidistributive \cite[Lemme~9]{Poly-Barbut}.)
In Theorem~\ref{perm can join}, below, we describe canonical join representations of permutations explicitly, in particular giving an independent proof that the weak order on permutations is semidistributive.
The argument in \cite{DuCh} is very different, but a proof that is similar in spirit to the proof here can be obtained by combining \cite[Proposition~6.4]{con_app} with special cases of \cite[Lemma~3.5]{shardint} and \cite[Theorem~3.6]{shardint},

A \newword{descent} in a permutation is a pair $x_i$, $x_{i+1}$ of adjacent entries such that $x_i>x_{i+1}$.
Given a permutation $x_1\cdots x_n$ and a descent $x_i>x_{i+1}$, define $\lambda(x,i)$ to be the permutation given by $1,\ldots,(x_{i+1}-1)$, then the values $\set{x_j:j<i,\,x_{i+1}<x_j<x_i}$ in increasing order, then $x_i$, then $x_{i+1}$, then $\set{x_j:i+1<j,\,x_{i+1}<x_j<x_i}$ in increasing order, and finally the values $(x_i+1),\ldots,n$.
Thus $\lambda(x,i)$ is join-irreducible and covers the permutation obtained by swapping the adjacent entries $x_i$ and $x_{i+1}$ in $\lambda(x,i)$.
Examples of the construction of $\lambda(x,i)$ occur below as part of Examples~\ref{can ex 1} and~\ref{can ex 2}.

\begin{prop}\label{min below with inv}
The permutation $\lambda(x,i)$ is the unique minimal element of the set $\set{y\le x:(x_i,x_{i+1})\in\inv(y)}$.
\end{prop}
\begin{proof}
The inversion set of $\lambda(x,i)$ consists of all pairs $(b,a)$ with ${x_i\ge b>a\ge x_{i+1}}$, with $b\in\set{x_j:1\le j\le i}$, and with $a\in\set{x_j:i+1\le j\le n}$.
Each of these is an inversion of $x$, so $\lambda(x,i)\le x$.
Given a permutation $y$ with $y\le x$ and $(x_i,x_{i+1})\in\inv(y)$ and such a pair $(b,a)$, we will show that $(b,a)$ is an inversion of $y$.
The elements $b$, $a$, $x_i$ and $x_{i+1}$ occur in the order $\cdots b\cdots x_ix_{i+1}\cdots a\cdots$ in $x$.
Also $(x_i,x_{i+1})\in\inv(y)$, or in other words $x_i$ precedes $x_{i+1}$ in $y$.
If $(b,a)\not\in\inv(y)$, or in other words if $a$ precedes $b$ in $y$, then necessarily either $a$ precedes $x_{i+1}$ or $b$ follows $x_i$ in $y$, or both.
Thus either $(a,x_{i+1})$ or $(x_i,b)$ is in $\inv(y)$, contradicting the fact that $y\le x$.
We have shown that $\inv(\lambda(x,i))\subseteq y$, and thus ${\lambda(x,i)\le y}$.
\end{proof}

The next theorem says in particular that the canonical joinands of a permutation are in bijection with the descents of the permutation.

\begin{theorem}\label{perm can join}
The canonical join representation of a permutation $x$ is 
\[x=\Join\set{\lambda(x,i):x_i>x_{i+1}}.\]
\end{theorem}
\begin{proof}
We first show that $\Join\set{\lambda(x,i):x_i>x_{i+1}}=x$.
Proposition~\ref{min below with inv} implies in particular that $\lambda(x,i)\le x$ for each $i$ with $x_i>x_{i+1}$, so ${\Join\set{\lambda(x,i):x_i>x_{i+1}}\le x}$.
If $\Join\set{\lambda(x,i):x_i>x_{i+1}}<x$, then some $y\lessdot x$ has $y\ge\Join\set{\lambda(x,i):x_i>x_{i+1}}$.
The permutation $y$ is obtained from $x$ by swapping $x_i$ and $x_{i+1}$ for some $i$ such that $x_i>x_{i+1}$.
But then $(x_i,x_{i+1})$ is not an inversion of $y$, while it is an inversion of $\lambda(x,i)$, and therefore $y\not\ge\lambda(x,i)$, and this is a contradiction.
We conclude that $\Join\set{\lambda(x,i):x_i>x_{i+1}}=x$.

We next show that $\set{\lambda(x,i):x_i>x_{i+1}}$ is an antichain.
Suppose $i$ and $j$ are distinct indices with $x_i>x_{i+1}$ and $x_j>x_{j+1}$.
The inversion set of $\lambda(x,i)$ is described in the proof of Proposition~\ref{min below with inv}.
In particular, if $(x_j,x_{j+1})$ is an inversion of $\lambda(x,i)$, then $j\le i$ and $i+1\le j+1$, which is impossible because $i\neq j$.
Therefore $(x_j,x_{j+1})$ is not an inversion of $\lambda(x,i)$.
Since $(x_j,x_{j+1})$ \emph{is} an inversion of $\lambda(x,j)$, we see that $\lambda(x,j)\not\le\lambda(x,i)$.
Thus $\set{\lambda(x,i):x_i>x_{i+1}}$ is an antichain.

Finally, we show that $\set{\lambda(x,i):x_i>x_{i+1}}\ll T$ whenever $x=\Join T$.
Suppose to the contrary that for some $i$ with $x_i>x_{i+1}$, there is no element $t\in T$ with $\lambda(x,i)\le t$.
Then Proposition~\ref{min below with inv} implies that there is no element $t\in T$ with $(x_i,x_{i-1})\in\inv(t)$.
Let $y$ be the permutation obtained from $x$ by swapping $x_i$ and $x_{i+1}$.
Then $\inv(y)=\inv(x)\setminus\set{(x_i,x_{i+1})}$.
But since the weak order is containment of inversion sets and since $x$ is an upper bound for $T$, we see that $y$ is also an upper bound for $T$, contradicting the supposition that $x=\Join T$.
\end{proof}

\begin{example}\label{can ex 1}
In Figure~\ref{weak ex}, we see directly that the canonical joinands of the permutation $3421$ are $2134$ and $1342$, in agreement with Theorem~\ref{perm can join}.
\end{example}

\begin{example}\label{can ex 2}
Let $x$ be the permutation $157842936\in S_9$.
Then $x$ has descents $8>4$ and $4>2$ and $9>3$.
The canonical joinands of $x$ are $\lambda(x,4)=123578469$, $\lambda(x,5)=142356789$ and $\lambda(x,7)=124578936$.
\end{example}

Theorem~\ref{perm can join} describes the join representation of a given permutation, but does not answer the following natural question.
\begin{question}\label{can q}
Which sets of join-irreducible permutations are canonical join representations?
\end{question}
In the next section, we answer Question~\ref{can q} using noncrossing arc diagrams.

\section{Canonical join representations and noncrossing arc diagrams}\label{proof sec}
In this section, we give a bijection between permutations and noncrossing arc diagrams by showing that noncrossing arc diagrams are a combinatorial model for canonical join representations.

A permutation is join-irreducible if and only if it has exactly one descent.
Suppose $x_1\cdots x_n$ is a join-irreducible permutation with descent $x_i>x_{i+1}$.
The unique permutation covered by $x$ is obtained by swapping the adjacent entries $x_i$ and $x_{i+1}$.
In particular, $x$ itself is the unique minimal element of $\set{y\le x:(x_i,x_{i+1})\in\inv(y)}$, so Proposition~\ref{min below with inv} says that $x=\lambda(x,i)$.
Thus, given that $x$ is join-irreducible, it is determined uniquely by the values $x_i$ and $x_{i+1}$ forming its unique descent and by the set of values $\set{x_j:j<i,\,x_{i+1}<x_j<x_i}$.
(To specify $x$, we don't need to specify $i$ explicitly; the two values and the set are enough.)

On the other hand, an arc satisfying (A) is determined by its endpoints and by the set of points to its left.
Thus there is a bijection between join-irreducible permutations and arcs satisfying (A).
The bijection sends a join-irreducible permutation $x$ to the arc connecting $a$ and $b$, where $b>a$ is the unique descent of $x$, and for each $c$ with $a<c<b$, having $c$ to the left of the arc if and only if $c$ occurs in $x$ to the left of the descent $ba$.
We extend this bijection to a map from permutations to sets of arcs.
Specifically, let $\delta$ be the map taking a permutation to the set of arcs corresponding to the elements of its canonical join representation. 
The following theorem constitutes a bijective proof of Theorem~\ref{diagrams perm}.

\begin{theorem}\label{main}
The map $\delta$ is a bijection from $S_n$ to the set of noncrossing arc diagrams on $n$ points.
\end{theorem}

As a first step in proving the theorem, we recast the map $\delta$ in way that makes it clear that it maps permutations to noncrossing arc diagrams.
Given a permutation $x=x_1,\ldots,x_n$, write each entry $x_i$ at the point $(i,x_i)$ in the plane.
For each descent $x_i>x_{i+1}$, draw a line from $x_i$ to $x_{i+1}$ as illustrated in the left picture of Figure~\ref{map to ncs} for the permutation $157842936$.
(Cf. Example~\ref{can ex 2}.)
\begin{figure}
\scalebox{0.9}{\includegraphics{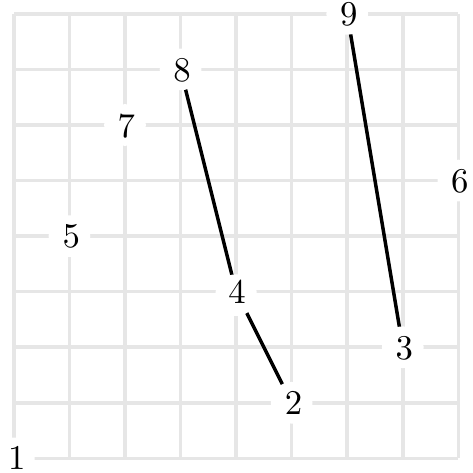}}\qquad\qquad\quad\scalebox{0.9}{\includegraphics{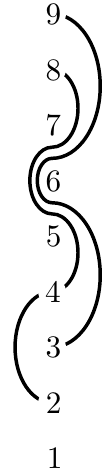}}
\caption{The map from permutations to noncrossing arc diagrams}
\label{map to ncs}
\end{figure}
Then move all of the numbers into a single vertical line, allowing the lines connecting descents to curve but not to pass through any of the numbers.
These lines become the arcs in a noncrossing arc diagram.
By Theorem~\ref{perm can join}, we see that these arcs correspond to the elements of the canonical join representation of~$x$, so the noncrossing arc diagram obtained is exactly~$\delta(x)$.

It will be convenient to prove Theorem~\ref{main} together with the following result, which was promised in the introduction.

\begin{prop}\label{pairwise works}
Given any collection of pairwise compatible arcs, there is a noncrossing arc diagram whose arcs are combinatorially equivalent to the given arcs.
\end{prop}

\begin{proof}[Proof of Theorem~\ref{main} and Proposition~\ref{pairwise works}]
Since a permutation is uniquely determined by its canonical join representation, the map $\delta$ is injective.

Let $\E$ be some collection of pairwise compatible arcs, each satisfying condition (A) of the definition of noncrossing arc diagrams.
We do not yet know that we can draw all the arcs of $\E$ together in such a way that condition (C1) holds.
However, we know that (A) holds for each arc in $\E$ and that (C1) and (C2) hold for each pair.
Consider the graph $\G$ defined on the given $n$ points with edges given by $\E$.
By (C2), each connected component of $\G$ is either an isolated point or a sequence $i_1>\cdots>i_k$ such that each $i_j$ and $i_{j+1}$ are connected by an arc in $\E$.

In Section~\ref{diagrams sec}, in connection with the definition of compatibility of arcs, we described combinatorial conditions that would require one arc to be drawn to the left or right of another.
Say a component $C_1$ of $\G$ is \newword{left of} another component $C_2$ (or equivalently that $C_2$ is \newword{right of} $C_1$) if there exist arcs $\alpha_1$ in $C_1$ and $\alpha_2$ in $C_2$ such that $\alpha_1$ must be drawn to the left of $\alpha_2$.
We claim that the relation ``is left of'' is acyclic on components of $\G$.
There are no 1-cycles in the relation.
A 2-cycle in the relation is a pair $C_1,C_2$ of components such that $C_1$ is both right and left of $C_2$.
If $C_2$ is left of $C_1$, then there exists a point $p$ that is left of (or an endpoint of) an arc in $C_1$ and is right of (or an endpoint of) an arc in $C_2$, with $p$ not an endpoint of both arcs.
At $p$, there is an arc $\alpha_1$ of $C_1$ that must be right of some arc $\alpha_2$ in $C_2$.
As one moves vertically upwards from $p$, there may be endpoints of arcs in $C_1$ or $C_2$.
As we pass through such an endpoint (say in $C_1$) and pass to a new arc in $C_1$, condition (C2) says that we don't also pass through an endpoint of $C_2$.  
We pass to a new arc $\alpha_1'$ in $C_1$ which also must be right of $\alpha_2$.  
Thus, continuing upwards and passing endpoints of arcs in $C_1$ or $C_2$, we find that we must continue to draw $C_1$ right of $C_2$.
After also making the same argument moving downward from $p$, we see that \emph{all} of $C_1$ can be drawn to the right of $C_2$.
In particular, $C_1$ is not left of $C_2$, so $C_1,C_2$ is not a 2-cycle. 

Now suppose that $k\ge 3$ and $C_1,\ldots,C_k$ are components of $\G$ with the property that $C_{i+1}$ is left of $C_i$ for each $i=1,\ldots k$.
We will index cyclically throughout the argument, so that in particular $C_1$ is left of $C_k$.
Each $C_i$ has some lowest vertex $a_i$ and some highest vertex $b_i$.
First, suppose the interval $[a_i,b_i]$ is contained in the interval $[a_{i+1},b_{i+1}]$ for some $i$.
Arguing as in the proof that $2$-cycles don't exist, we see that all of $C_i$ must be drawn to the right of $C_{i+1}$ to satisfy (C2).
There is a point $p$ that is left of (or an endpoint of) some arc in $C_{i-1}$ and right of (or an endpoint of) some arc in $C_i$.
The point $p$ is also right of $C_{i+1}$, and we conclude that $C_{i+1}$ is left of $C_{i-1}$.
Thus we obtain a $(k-1)$-cycle by removing $C_i$ from $C_1,\ldots,C_k$.
If instead $[a_i,b_i]$ contains $[a_{i+1},b_{i+1}]$, the we similarly make a $(k-1)$-cycle.

Now suppose that for all $i$, there is no containment relation between $[a_i,b_i]$ and $[a_{i+1},b_{i+1}]$.
Since $C_1,\ldots,C_k$ is a cycle, there exists $i$ with ${a_{i-1}<a_i<b_{i-1}<b_i}$ and ${a_{i+1}<a_i<b_{i+1}<b_i}$.
Then the point $a_i$ is to the left of some arc of $C_{i-1}$ and to the right of some arc of $C_{i+1}$.
Thus again we obtain a $(k-1)$-cycle by removing $C_i$ from $C_1,\ldots,C_k$.

We have shown that, for $k\ge3$, if there exists a $k$-cycle in the ``is left of'' relation, then there exists a $(k-1)$-cycle.
Since $1$-cycles and $2$-cycles don't exist, this completes the proof of the claim that the relation is acyclic.

We now use the claim to recursively construct a permutation $x=x_1\cdots x_n$ with the following properties:
\begin{enumerate}
\item[(i)]If $i>j$, then $i$ and $j$ form a descent in $x$ if and only if $i$ and $j$ are connected by an arc in $\E$.
\item[(ii)]If $i>j>k$ and $i$ and $k$ are connected by an arc in $\E$, then $j$ is to the left of $i$ in $x$ if and only if $j$ is to the left of the arc connecting $i$ to $k$.
(Equivalently, $j$ is right of $k$ in $x$ if and only if $j$ is right of the arc.)
\end{enumerate}

The claim implies that $\G$ has at least one \newword{left component}, meaning a component that is not right of any other component of $\G$.
Since each left component has nothing to its left, the left components can be totally ordered from smaller-valued endpoints to higher-valued ones.
Take the smallest-valued left component, delete it from $\G$ and write its labels in decreasing order.
(When we delete the component, we keep the original labels on the remaining points of $\G$.)
By the claim, if the remaining graph is nonempty, then it has at least one left component, and in particular may have additional left components that were not left components in the original diagram.
We again take the smallest left component, delete it, and write its labels in decreasing order.
When all of the points of the diagram have been deleted, the result is a permutation.

The output of the process is a permutation that satisfies condition (ii) above by construction.
Because the arcs in $\E$ satisfy (C2) pairwise, each arc in the diagram becomes a descent in the permutation.
It is also easy to see that there are no other descents:
If there were another descent, that would mean that at some step the process deletes one left component $C_1$, and then another left component $C_2$ whose highest point is lower than the lowest point of $C_1$.
But then $C_2$ was already a left component before $C_1$ was deleted, so $C_2$ should have been deleted before $C_1$.
By this contradiction, we see that the output permutation $x$ satisfies (i).

Conditions (i) and (ii) imply that $\delta(x)$ is a noncrossing arc diagram whose arcs are exactly $\E$.
In particular, we have proved Proposition~\ref{pairwise works}.
We have also shown that $\delta$ is surjective, thus completing the proof that it is a bijection.
\end{proof}

The inverse to $\delta$ is given by the recursive process described in the proof above, as exemplified in Figure~\ref{inversefig}.
At each step, the left components of the remaining diagram are shown in red (or gray if the figure is not viewed in color).
\begin{figure}
\begin{tabular}{|l|l|l|l|l|l|l|}
\hline&&&&&&\\[-8pt]
Step&Start&1&2&3&4&5\\[2 pt]\hline&&&&&&\\[-8pt]
Permutation so far&&4&46&46731&4673152&46731528\\[2 pt]\hline&&&&&&\\[-8pt]
\raisebox{59pt}{Diagram remaining}&\scalebox{0.85}{\includegraphics{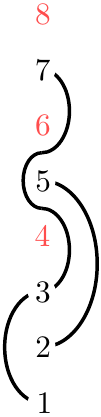}}&\scalebox{0.85}{\includegraphics{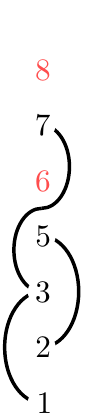}}&\scalebox{0.85}{\includegraphics{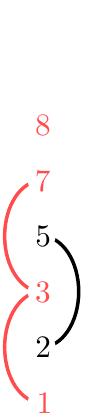}}&\scalebox{0.85}{\includegraphics{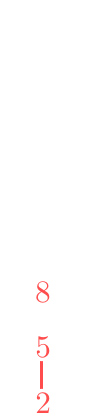}}&\scalebox{0.85}{\includegraphics{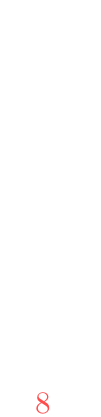}}&\\[2 pt]\hline
\end{tabular}
\caption{The map from noncrossing arc diagrams to permutations}
\label{inversefig}
\end{figure}

\begin{remark}\label{foreshadowing}
Noncrossing arc diagrams and the map $\delta$ are foreshadowed in work~\cite{shardint} on \newword{shard intersections} and particularly in work of Bancroft~\cite{Bancroft} and Petersen~\cite{Petersen} on shard intersections in type A. 
See in particular \cite[Theorem~3.6]{shardint}, \cite[Proposition~4.7]{shardint}, \cite[Section~3]{Bancroft}, and \cite[Section~2.1]{Petersen}.
\end{remark}

The noncrossing arc diagrams are the faces of a simplicial complex whose vertices are the arcs.
The face corresponding to a noncrossing arc diagram is the set of arcs appearing in the diagram.
The following is an immediate corollary of Theorem~\ref{main}.

\begin{cor}\label{main can join cplx}
The simplicial complex of noncrossing arc diagrams on $n$ points is isomorphic to the canonical join complex of the weak order on $S_n$.
The isomorphism is induced by the map from arcs to join-irreducible elements.
\end{cor}

Figure~\ref{arc complex} shows the canonical join complex of the weak order on $S_4$, in the form of the complex of noncrossing arc diagrams.
\begin{figure}
\includegraphics{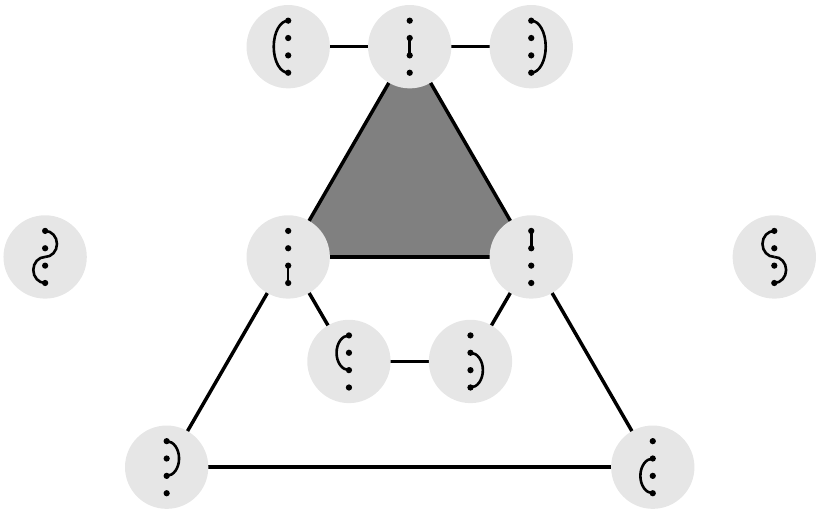}
\caption{The canonical join complex of $S_4$}
\label{arc complex}
\end{figure}

Theorem~\ref{main} and Proposition~\ref{pairwise works} also lead to an answer to Question~\ref{can q}.
In the following corollary, the equivalence of (i) and (ii) is immediate by Theorem~\ref{main} and the equivalence of these with (iii) and (iv) follows by Proposition~\ref{pairwise works}.

\begin{cor}\label{can a}
Suppose $J$ is a set of join-irreducible elements of $S_n$ and suppose $\E$ is the corresponding collection of arcs.
The following are equivalent.
\begin{enumerate}
\item[\textnormal{(i)}] $J$ is the canonical join representation of a permutation.
\item[\textnormal{(ii)}] There is a noncrossing arc diagram whose arcs are combinatorially equivalent to the arcs in $\E$.
\item[\textnormal{(iii)}] The arcs in $\E$ are pairwise compatible.
\item[\textnormal{(iv)}] Each $2$-element subset of $J$ is the canonical join representation of a permutation.
\end{enumerate}
\end{cor}

A simplicial complex is called \newword{flag} it it is the clique complex of its $1$-skeleton.
Equivalently, it is flag if, for every subset $S$ of the vertices not forming a face, there is a pair of distinct elements of $S$ not forming an edge in the complex.
The following corollary is immediate by Corollary~\ref{can a}.

\begin{cor}\label{arcs flag}
The canonical join complex of the weak order on permutations is flag.
\end{cor}

%

Theorem~\ref{main} also immediately implies some additional counting results, and combines with known results to prove others.

Let $\eulerian{n}{k}$ denote the Eulerian number, the number of permutations of $\set{1,\ldots,n}$ with exactly $k$ descents.
By Theorem~\ref{perm can join}, the descents of a permutation are in bijection with the join-irreducible permutations in its canonical join representation.
The latter are in bijection with the arcs in the corresponding noncrossing arc diagram, so we have the following theorem.
\begin{theorem}\label{diagrams eulerian}
The number of noncrossing arc diagrams on $n$ points with exactly $k$ arcs is $\eulerian{n}{k}$.
\end{theorem}

Say a noncrossing arc diagram is a \newword{matching} if all of its arcs are disjoint, even at their endpoints.
In this case, the diagram is a planar representation of a matching in the usual graph-theoretic sense.
A noncrossing arc diagram is a matching if and only if the corresponding permutation $x$ has no three consecutive entries $x_ix_{i+1}x_{i+2}$ with $x_i>x_{i+1}>x_{i+2}$.  
Such permutations are said to \newword{avoid the consecutive pattern $321$}.
The following theorem is an immediate consequences of Theorem~\ref{main}.
\begin{theorem}\label{diagrams matchings}
The noncrossing arc diagrams on $n$ points that are matchings are in bijection with permutations avoiding the consecutive pattern $321$.  
\end{theorem}

The exponential generating function for permutations avoiding the consecutive pattern $321$ (and thus for noncrossing arc diagrams that are matchings) is determined in~\cite[Theorem~4.1]{ElizaldeNoy}.

Say a noncrossing arc diagram is a \newword{perfect matching} if it is a matching and each point is the endpoint of an arc.
An \newword{alternating permutation}\footnote{Conventions vary:  Sometimes the term alternating permutation refers instead to permutations satisfying $x_1<x_2>x_3<x_4>\cdots >x_{2n-1}<x_{2n}$, and sometimes the term refers to permutations satisfying either of the two conditions.} in $S_{2n}$ is a permutation $x$ with $x_1>x_2<x_3>x_4<\cdots <x_{2n-1}>x_{2n}$.
Alternating permutations in $S_{2n}$ are characterized by avoiding the consecutive pattern $321$ and having exactly $n$ descents.
Thus $\delta$ restricts to the bijection described in the following theorem.

\begin{theorem}\label{diagrams perfect}
Alternating permutations in $S_{2n}$ are in bijection with noncrossing arc diagrams on $2n$ points that are perfect matchings.
\end{theorem}

The exponential generating function for alternating permutations (and thus for noncrossing arc diagrams that are perfect matchings) is $\sec x$.

Left noncrossing perfect matchings in the sense of this paper are also well-known, usually under the name ``noncrossing matchings.''
Restricting $\delta$, these are in bijection with $231$-avoiding alternating permutations.
From \cite[Theorem~2.2]{Mansour} (for the permutations) or from \cite[Exercise~6.19(o)]{EC1} (for the matchings), we obtain the following enumeration:
The number of left noncrossing arc diagrams on $2n$ points that are perfect matchings is the Catalan number $C_n=\frac1{n+1}\binom{2n}n$.

\section{Noncrossing arc diagrams and lattice quotients of the weak order}\label{quot sec}
In this section, we discuss how restricted classes of noncrossing arc diagrams arise from lattice quotients of the weak order modulo lattice congruences.
We begin by briefly reviewing some background on lattice congruences.
We then describe how noncrossing arc diagrams are a convenient combinatorial model for lattice quotients of the weak order.
Finally, we give some examples of restricted classes of noncrossing arc diagrams arising from quotients.

A \newword{congruence} on a lattice $L$ is an equivalence relation on $L$ that respects the meet and join operations.
That is, if $x_1\equiv y_1$ and $x_2\equiv y_2$, then $x_1\meet x_2\equiv y_1\meet y_2$ and $x_1\join x_2\equiv y_1\join y_2$.
The property of respecting meets and joins is equivalent, \emph{for finite lattices}, to the following three conditions:
First, equivalence classes are intervals in the lattice.
Second, the map $\pidown^\Theta$ taking each element to the bottom element of its equivalence class is order-preserving.
Third, the map $\piup_\Theta$ taking each element to the top element of its equivalence class is order-preserving.

The \newword{quotient} of $L$ modulo a congruence $\Theta$ is the lattice $L/\Theta$ whose elements are the congruence classes with the meet or join of classes defined by taking the meet or join of representatives.
That is, if $C$ and $D$ are congruence classes with $x\in C$ and $y\in D$, then $C\meet D$ is the class of $x\meet y$ and $C\join D$ is the class of $x\join y$.
\emph{For finite lattices}, the quotient is isomorphic as a poset to the subposet $\pidown^\Theta(L)$ of $L$ induced by the elements that are the bottom elements of their congruence classes. 
(The subposet $\pidown^\Theta(L)$ is always a join-sublattice of $L$ but can fail to be a sublattice of~$L$.)
This way of thinking about the quotient suggests that we think of ``contracting'' each congruence class onto its bottom element.
An element of $L$ is \newword{contracted} by $\Theta$ if it is congruent, modulo $\Theta$, to some element below it.
Thus an element is \newword{uncontracted} if and only if it is at the bottom of its congruence class, and so $\pidown^\Theta(L)$ is the subposet of $L$ induced by uncontracted elements.

In a lattice $L$ where each element has a canonical join representation, an element is contracted by $\Theta$ if and only if one or more of its canonical joinands is contracted by $\Theta$.
In particular, the join-irreducible elements of $\pidown^\Theta(L)$ are exactly the join-irreducible elements of $L$ not contracted by $\Theta$.
Furthermore, the canonical join-representation of an element of the quotient $\pidown^\Theta(L)$ coincides with its canonical join-representation in $L$.
Thus Theorem~\ref{main} implies the following theorem.

\begin{theorem}\label{quot diagram}
Given a lattice congruence $\Theta$ on the weak order on $S_n$, the elements of the quotient lattice $\pidown^\Theta(S_n)$ are in bijection with the noncrossing arc diagrams on $n$ points consisting only of arcs corresponding to join-irreducible elements not contracted by $\Theta$.
The bijection maps an element of $\pidown^\Theta(S_n)$ to the set of arcs corresponding to its canonical join-representation (in $\pidown^\Theta(S_n)$ or in $S_n$).
\end{theorem}

The canonical join complex of $\pidown^\Theta(L)$ is isomorphic to the subcomplex of the canonical join complex of $L$ induced by the vertices (join-irreducible elements of $L$) not contracted by $\Theta$.
Thus we have the following additional corollaries.  
(Compare Corollaries~\ref{main can join cplx} and~\ref{arcs flag}.)

\begin{cor}\label{quot can join cplx}
Given a lattice congruence $\Theta$ on the weak order on $S_n$, the canonical join complex of $\pidown^\Theta(S_n)$ is isomorphic to the simplicial complex of noncrossing arc diagrams on $n$ points using only arcs corresponding to join-irreducible elements not contracted by $\Theta$.
The isomorphism is induced by the map from join-irreducible elements to arcs.
\end{cor}

\begin{cor}\label{arcs quot flag}
For any lattice congruence $\Theta$ on the weak order on $S_n$, the canonical join complex of $S_n/\Theta$ is flag.
\end{cor}

The fact that an element is contracted if and only if one or more of its canonical joinands is contracted implies in particular that the congruence is determined uniquely by which join-irreducible elements it contracts.
As one might expect, the join-irreducible elements cannot be contracted independently.  
Instead, there is a pre-order on join-irreducible elements such that a set of join-irreducible elements is contracted by some congruence if and only if that set is closed under ``going down'' in the preorder.
We refer to this preorder as \newword{forcing} and describe it in words rather than notation.
A join-irreducible element $j_1$ is above $j_2$ in the forcing order if every congruence contracting $j_1$ also contracts $j_2$.
In this case, we say that $j_1$ \newword{forces} $j_2$.

The forcing preorder on join-irreducible elements of the weak order on $S_n$ was worked out in \cite[Section~8]{congruence}.
In particular, in the weak order on $S_n$, the forcing preorder was shown to be a partial order (i.e.\ it has no directed cycles).
The forcing order on join-irreducible permutations has a nice description in terms of arcs satisfying condition (A).
We will phrase the description by saying that one arc forces another arc, meaning that the forcing relation holds on the corresponding join-irreducible permutations.

Suppose $\alpha_1$ and $\alpha_2$ are arcs satisfying (A), with $\alpha_1$ connecting points $p_1$ and $q_1$ with $p_1<q_1$ and with $\alpha_2$ connecting $p_2$ and $q_2$ with $p_2<q_2$.
We say $\alpha_1$ is a \newword{subarc} of $\alpha_2$ if
\begin{enumerate}
\item[(i) ]$p_2\le p_1<q_1\le q_2$, and 
\item[(ii) ]The set of points left of $\alpha_1$ equals the set of points in $\set{p_1+1,\ldots,q_1-1}$ left of $\alpha_2$.
\end{enumerate}
Less formally, to construct a subarc of an arc $\alpha$, we choose two distinct horizontal lines, each passing through one of the $n$ points and each intersecting $\alpha$, possibly at an endpoint of $\alpha$.
We construct the subarc by cutting $\alpha$ along the two lines and retaining the middle portion of $\alpha$.
Each endpoint of this middle section is attached to one of the $n$ points, specifically, the point at the same height.
(Possibly the endpoint of the middle section is already one of the $n$ points, if the line cuts $\alpha$ at an endpoint.)
This process is illustrated in Figure~\ref{subarcfig}.
\begin{figure}
\begin{tabular}{ccccc}
\scalebox{0.8}{\includegraphics{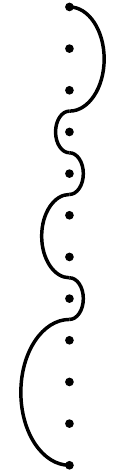}}&
\scalebox{0.8}{\includegraphics{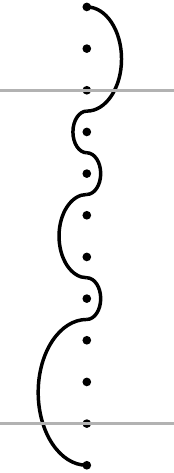}}&
\scalebox{0.8}{\includegraphics{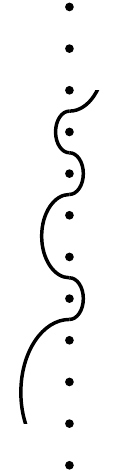}}&
\scalebox{0.8}{\includegraphics{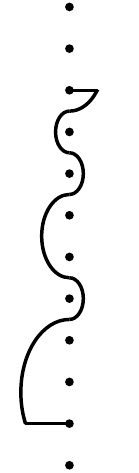}}&
\scalebox{0.8}{\includegraphics{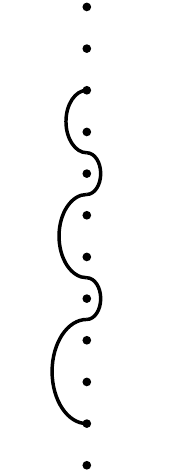}}\\
\end{tabular}
\caption{Constructing a subarc}
\label{subarcfig}
\end{figure}

In \cite[Section~8]{congruence}, the join-irreducible permutations are encoded as subsets as explained in \cite[Section~5]{congruence}.
Translating this encoding into the language of arcs and subarcs, either \cite[Theorem~8.1]{congruence} or \cite[Theorem~8.2]{congruence} immediately implies the following theorem:

\begin{theorem}\label{arc forcing}
An arc $\alpha_1$ forces an arc $\alpha_2$ if and only if $\alpha_1$ is a subarc of $\alpha_2$.
\end{theorem}

Figure~\ref{forcingfig} shows the forcing order on join-irreducible permutations in $S_4$, represented by arcs.
\begin{figure}
\includegraphics{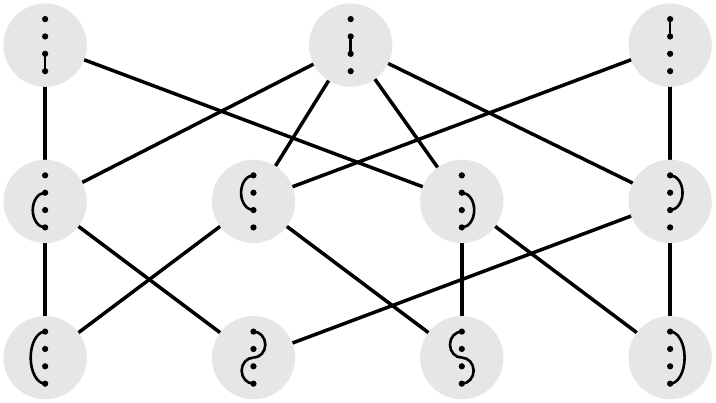}
\caption{The forcing order on arcs for $n=4$}
\label{forcingfig}
\end{figure}

Theorems~\ref{quot diagram} and~\ref{arc forcing} let us understand quotients of the weak order on $S_n$ entirely in terms of noncrossing arc diagrams.
(It is not immediately apparent how to realize the partial order/lattice structure on the quotient in terms of noncrossing arc diagrams, so we confine ourselves to statements about the quotients as sets.)
The following corollary is immediate by Theorem~\ref{arc forcing} (for the first assertion), then Theorem~\ref{quot diagram} (for the second and third assertions).

\begin{cor}\label{closed subarcs}
A set $U$ of arcs corresponds to the set of \emph{un}contracted join-irreducible permutations of some congruence $\Theta$ if and only if $U$ is closed under passing to subarcs.
In this case, the map $\delta$ restricts to a bijection from permutations not contracted by $\Theta$ to arc diagrams consisting only of arcs in $U$.
For each $k$, the map $\delta$ further restricts to a bijection between uncontracted permutations with exactly $k$ descents and arc diagrams consisting of exactly $k$ arcs, all of which are in~$U$.
\end{cor}

We now give an explicit description of the uncontracted permutations.
Suppose $a$ and $b$ are integers with $1\leq a<b\leq n$ and suppose $R\subseteq\set{(a+1),\ldots,(b-1)}$.
Write $L=\set{(a+1),\ldots,(b-1)}\setminus R$.
A permutation $x$ \newword{has a $(b,a,R)$-pattern} if $(x_i,x_{i+1})$ is a descent of $x$, if $x_i\geq b$ and $x_{i+1}\leq a$, and if all of the elements of $L$ appear in $x$ \emph{to the left of} $x_i,x_{i+1}$ and all of the elements of $R$ appear in $x$ \emph{to the right of} $x_i,x_{i+1}$.
If $x$ has no $(b,a,R)$-pattern, then $x$ \newword{avoids $(b,a,R)$}.
The triple $(b,a,R)$ precisely specifies an arc $\alpha(b,a,R)$ connecting $a$ and $b$ and having the points in $R$ on its right while having the points in $L$ on its left.
The same information precisely specifies a join-irreducible permutation $\lambda(b,a,R)$ consisting of the entries $1,\ldots,(a-1)$, then $L$ in increasing order, then $b$, then $a$, then $R$ in increasing order, and finally $(b+1),\ldots,n$.
Our bijection between join-irreducible permutations and arcs sends $\lambda(b,a,R)$ to $\alpha(b,a,R)$.

\begin{cor}\label{perm quot}
Let $\Theta$ be the smallest congruence contracting the join-irreducible permutations in a given set $\set{\lambda(b_i,a_i,R_i):i\in I}$.
Then the permutations not contracted by $\Theta$ are exactly the permutations that avoid $(b_i,a_i,R_i)$ for every $i\in I$.
\end{cor}
\begin{proof}
By Corollary~\ref{closed subarcs}, a permutation $x$ is uncontracted by $\Theta$ if and only if no arc in $\delta(x)$ has an arc $\alpha(b_i,a_i,R_i)$ as a subarc.
By the definition of $\delta$, this description of the uncontracted permutations is a criterion on the canonical join-representations of the permutations.
Theorem~\ref{perm can join} translates the criterion into the requirement of avoiding $(b_i,a_i,R_i)$ for every $i\in I$.
\end{proof}

The results of \cite{con_app} imply a useful enumerative statement.
The Hasse diagram of the weak order on permutations is dual to the simplicial fan $\F$ defined by the Coxeter arrangement of type A.
Given a congruence $\Theta$ on the weak order, for each congruence class, one can ``glue'' together the cones corresponding to the elements of the congruence class.
In \cite[Theorem~1.1]{con_app}, it is shown, among other things, that for each congruence class, the result of the gluing is a convex cone, that these glued cones form a fan $\F_\Theta$, and that any linear extension of the quotient lattice defines a shelling order on $\F_\Theta$.
When $\F_\Theta$ is simplicial, its $h$-vector (i.e.\ the $h$-vector of the corresponding simplicial complex) counts permutations not contracted by $\Theta$ according to their number of descents.
The fan $\F_\Theta$ is simplicial if and only if the Hasse diagram of the quotient is a regular graph.
By Corollary~\ref{closed subarcs}, we have the following corollary.

\begin{cor}\label{simp h}
Suppose $\Theta$ is a congruence on the weak order on permutations and suppose $U$ is the set of arcs corresponding to join-irreducible permutations not contracted by $\Theta$.
If the fan $\F_\Theta$ is simplicial, then the entry $h_k$ in the $h$-vector of $\F_\Theta$ counts noncrossing arc diagrams consisting of exactly $k$ arcs, all of which are in $U$.
\end{cor}

We conclude with some examples.

\begin{example}[Permutations with restricted size of descent]
The \newword{length} of an arc is $1$ plus the number of points the arc passes either left or right of.  
If the points are evenly spaced at unit distance, then the length of an arc is the distance between its endpoints.
Thus the arcs in a noncrossing arc diagram on $n$ points have lengths $1$ through $n-1$.
Fixing some $k\ge 1$, the set of arcs of length less than~$k$ is closed under passing to subarcs, so there is a lattice quotient of the weak order on $S_n$ corresponding to noncrossing arc diagrams with arcs of length less than~$k$.
The corresponding congruence is the smallest congruence that contracts all join-irreducible permutations corresponding to arcs of length~$k$.
The uncontracted permutations $x_1x_2\cdots x_n$ are characterized by the requirement that $x_i-x_{i+1}<k$ for $i=1,\ldots n-1$.
These permutations are also easily counted by induction on~$n$, and we see that the number of noncrossing arc diagrams on $n$ points with arcs of length less than~$k$ is $\prod_{i=1}^n\min(i,k)$.
This is $n!$ for $n\le k$ and $k!k^{n-k}$ for $k\le n$.
\end{example}

\begin{example}[The Tamari lattice and Cambrian lattices of type A]
The set of left arcs is closed under passing to subarcs.
Thus the permutations corresponding to left noncrossing arc diagrams (as defined in the introduction) are a lattice quotient $S_n/\Theta$ of the weak order on $S_n$.
Then $\Theta$ is the smallest congruence contracting all join-irreducible permutations corresponding to right arcs of length $2$.
Thus by Corollary~\ref{perm quot}, permutations not contracted by $\Theta$ are exactly the permutations that avoid $231$ in the usual sense.
We see that $S_n/\Theta$ is the Tamari lattice.

More generally, arbitrarily designate each of the $n$ points either as a right point or a left point.
Consider the set $U$ of arcs that do not pass to the right of any right point and do not pass to the left of any left point.
This set is closed under passing to subarcs.
Let $\Theta$ be the congruence that leaves uncontracted exactly the join-irreducible permutations corresponding to arcs in $U$.
Then $\Theta$ is the smallest congruence contracting the join-irreducible permutations corresponding to arcs of length $2$ that pass right of a right point or left of a left point.
The quotient $S_n/\Theta\cong \pidown^\Theta(S_n)$ is a \newword{Cambrian lattice} of type~A.
The subposet $\pidown^\Theta(S_n)$ has the special property that it is a sublattice of $S_n$.
The fan $\F_\Theta$ is simplicial, and in fact is the normal fan to an associahedron.
Corollary~\ref{simp h} reflects the well-known fact that the $h$-vector of the associahedron is given by the Narayana numbers.
For more information on the Tamari lattice and Cambrian lattice of type~A, see \cite[Sections~5--6]{cambrian}.
\end{example}

\begin{example}[Twisted Baxter permutations/diagrams with left and right arcs]
The set consisting of all left arcs and all right arcs is closed under passing to subarcs.
Thus the noncrossing arc diagrams having only left arcs and right arcs constitute a lattice quotient of the weak order.
This quotient is mentioned in \cite[Section~10]{con_app} and studied extensively in \cite{rectangle}.
Let $\Theta$ be the corresponding congruence.
The left and right arcs are exactly the arcs with no inflections (as defined in Section~\ref{diagrams sec}), so $\Theta$ is the smallest congruence contracting all join-irreducible permutations whose corresponding arcs have length $3$ and one inflection point.
For $n\ge 4$, the subposet $\pidown^\Theta(S_n)$ consisting of uncontracted permutations is not a sublattice of the weak order, and the corresponding fan is not simplicial.
(Both of these facts are easily verified for $n=4$ with the help of Figure~\ref{weak ex}.)

Corollary~\ref{perm quot} implies that $\pidown^\Theta(S_n)$ consist of the permutations having no subsequence $x_i\cdots x_jx_{j+1}\cdots x_k$ with $x_{j+1}<x_i<x_j$ and $x_{j+1}<x_k<x_j$.
These are the \newword{twisted Baxter permutations} of \cite[Section~10]{con_app} and \cite{rectangle}, which are known to be in bijection with the \newword{Baxter permutations} of \cite{CGHK}.
(This is an unpublished result of Julian West.  For a published proof, see \cite[Theorem~8.2]{rectangle} or combine \cite[Theorem~4.14]{giraudo} and \cite[Proposition~4.15]{giraudo}.
See also \cite[Section~2.4]{Dilks}.)
The Baxter permutations are counted \cite{CGHK} by the Baxter numbers, and Theorem~\ref{diagrams left right} follows.

Theorem~\ref{diagrams left right} is interesting in relation to the \newword{twin binary trees} of Dulucq and Guibert \cite{DGStack}.
These are pairs of planar binary trees that are complementary in a certain sense.
These twin trees can be seen when one cuts a diagonal rectangulation along its diagonal.
In a similar way, a noncrossing arc diagram composed of left arcs and right arcs is a pair of Catalan objects---a left noncrossing arc diagram and a right noncrossing arc diagram---that are compatible in the sense that their union is still a noncrossing arc diagram.
(There are two issues:  First, whether the left noncrossing arc diagram and the right noncrossing arc diagram have exactly the same set of arcs connecting adjacent vertices, and second, whether the union satisfies condition~(C2).)
\end{example}

\begin{example}[$2$-clumped permutations/generic rectangulations]
The diagrams whose arcs have at most one inflection point correspond to the \newword{$2$-clumped permutations} of \cite{generic}.
The latter are in bijection \cite[Theorem~4.1]{generic} with generic rectangulations with $n$ rectangles.
Theorem~\ref{diagrams inflection} follows.
More generally, diagrams consisting of arcs with at most $k$ inflection points correspond to the $(k+1)$-clumped permutations.
The $(k+1)$-clumped permutations are not well-studied for $k>1$.
\end{example}

\section*{Acknowledgments}
Thanks to Emily Barnard for helpful comments on several earlier versions of this paper, and in particular for pointing out that Proposition~\ref{pairwise works} needed to be argued.
Thanks also to Vic Reiner and David Speyer for helpful comments.


%



\begin{thebibliography}{12}
\bibitem{ABP}
E. Ackerman, G. Barequet, and R. Pinter, 
\textit{On the number of rectangulations of a planar point set.}
J. Combin. Theory Ser. A \textbf{113} (2006), no. 6, 1072--1091. 

\bibitem{Bancroft}
E. Bancroft, 
\textit{The shard intersection order on permutations.}
Preprint, 2011. (arXiv: 1103.1910)

\bibitem{CGHK}
F. R. K. Chung, R. L. Graham, V. E. Hoggatt Jr. and M. Kleiman,
\textit{The number of Baxter permutations.}
J. Combin. Theory Ser. A \textbf{24} (1978), no. 3, 382--394. 

\bibitem{Dilks}
K. Dilks,
\textit{Involutions on Baxter Objects.}
Preprint, 2014. (arXiv:1402.2961)

\bibitem{DGStack}
S. Dulucq and O. Guibert,
\textit{Stack words, standard tableaux and Baxter permutations.}
Discrete Math. \textbf{157} (1996), no. 1--3, 91--106. 

\bibitem{DuCh}
V. Duquenne and A. Cherfouh,
\textit{On permutation lattices.}
Math. Social Sci. \textbf{27} (1994), no. 1, 73--89. 

\bibitem{ElizaldeNoy}
S. Elizalde and M. Noy,
\textit{Consecutive patterns in permutations.}
Adv. in Appl. Math. \textbf{30} (2003), no. 1--2, 110--125. 

\bibitem{FreeLattices}
R. Freese, J. Je\v{z}ek, and J. Nation,
\textit{Free lattices.} 
Mathematical Surveys and Monographs, \textbf{42}. 
American Mathematical Society, Providence, RI, 1995.

\bibitem{giraudo}
S. Giraudo,
\textit{Algebraic and combinatorial structures on pairs of twin binary trees.}
J. Algebra \textbf{360} (2012), 115--157. 

\bibitem{rectangle}
S. Law and N. Reading,
\textit{The Hopf algebra of diagonal rectangulations.}
J. Combin. Theory Ser. A. \textbf{119} (2012) no. 3, 788--824.

\bibitem{Poly-Barbut}
C. Le Conte de Poly-Barbut,
\textit{Sur les treillis de Coxeter finis.}
Math. Inform. Sci. Humaines No. \textbf{125} (1994), 41--57. 

\bibitem{Mansour}
T. Mansour,
\textit{Restricted 132-alternating permutations and Chebyshev polynomials.}
Ann. Comb. \textbf{7} (2003), no. 2, 201--227. 

\bibitem{Petersen}
T. K. Petersen,
\textit{On the shard intersection order of a Coxeter group.}
SIAM J. Discrete Math. \textbf{27} (2013), no. 4, 1880--1912. 

\bibitem{congruence}
N. Reading,
\textit{Lattice congruences of the weak order.}
Order \textbf{21} (2004) no.~4, 315--344. 

\bibitem{con_app}
N.~Reading, 
\textit{Lattice congruences, fans and Hopf algebras.}
J. Combin. Theory Ser. A \textbf{110} (2005), no. 2, 237--273. 

\bibitem{cambrian}
N.~Reading, 
\textit{Cambrian Lattices.}
Adv. Math. \textbf{205} (2006), no.~2, 313--353.

\bibitem{shardint}
N. Reading, 
\textit{Noncrossing partitions and the shard intersection order.} 
 J. Algebraic Combin. \textbf{33} (2011), no. 4, 483--530.

\bibitem{generic}
N. Reading, 
\textit{Generic rectangulations.}
European J. Combin. \textbf{33} (2012) 610--623.

\bibitem{EC1}
R. Stanley,
\textit{Enumerative Combinatorics, Volume I.}
Cambridge Studies in Advanced Mathematics, \textbf{49},
Cambridge Univ. Press 1997.


\end{thebibliography}
\end{document}